\documentclass[12pt]{article}
\usepackage{amsmath,amsthm,amsfonts,hyperref,url}
\usepackage{times}
\usepackage{enumerate,tikz}

\newtheorem{theorem}{Theorem}

\newtheorem{lemma}[theorem]{Lemma}

\theoremstyle{definition}
\newtheorem{note}[theorem]{Note}
\newtheorem{dfn}[theorem]{Definition}
\newtheorem{rem}[theorem]{Remark}

\DeclareMathOperator{\Aut}{Aut}
\DeclareMathOperator{\rank}{rank}

\newcommand{\cD}{\mathcal{D}}
\newcommand{\cB}{\mathcal{B}}

\title{Ternary codes, biplanes, and the nonexistence of some 
quasi-symmetric and quasi-3 designs}

\author{Akihiro Munemasa\thanks{Graduate School of Information Sciences, Tohoku University, Sendai, 980-8579 Japan. email:munemasa@math.is.tohoku.ac.jp} and Vladimir D. Tonchev\thanks{Mathematical Sciences, Michigan Technological University, Houghton, Michigan, 49931 USA. email:tonchev@mtu.edu}}

\begin{document}
\maketitle

\begin{abstract}
The dual codes of the ternary linear codes of the residual designs
of biplanes on 56 points are used to prove the nonexistence
of quasi-symmetric 
2-$(56,12,9)$ and 2-$(57,12,11)$ designs with intersection
numbers 0 and 3, 
and the nonexistence of a 2-$(267,57,12)$ quasi-3 design.
The nonexistence of a 2-$(149,37,9)$ quasi-3 design is also proved. 
\medskip

{\bf Keywords:} linear code, quasi-symmetric design, symmetric design, 
residual design, biplane, quasi-3 design.

\end{abstract}

\section{Introduction}

We assume familiarity with basic facts and notions from
combinatorial design theory and coding theory
 (\cite{AK}, \cite{BJL}, \cite{CRC},  \cite{HP}, \cite{T88}).

A combinatorial {\it design} (or an {\it incidence structure}) 
is a pair
$\cD=(X, \cB)$ of a finite set $X=\{ x_i \}_{i=1}^v$
of {\it points},
and a collection $\cB=\{ B_j \}_{j=1}^b$ of subsets 
$B_j \subseteq X$, called {\it blocks}.
The (points by blocks) {\it incidence matrix} $A=( a_{i,j})$ of
a design $\cD$  
with $v$ points and $b$ blocks is a $(0,1)$-matrix with $v$
rows indexed by the points and $b$ columns indexed by the blocks, 
where $a_{i,j}=1$ if the $i$th point belongs to the $j$th block, and 
$a_{i,j}=0$ otherwise.
The transposed matrix $A^T$ is called the blocks by points incidence 
matrix of $\cD$.
If $p$ is a prime number, the $p$-rank of $A$
(or $\rank_{p}A$),  is defined as the rank of $A$ 
over a finite field of characteristic $p$.

Given a design $\cD$ with $v$ points  and $b$ blocks,
and a finite field $F=GF(q)$,
one can define two linear codes over $F$ associated with $\cD$:
the code of length $b$ spanned by the rows of the $v$ by $b$ 
incidence matrix $A$ is called the code of $\cD$ spanned by the points,
while the code of length $v$ spanned by the columns of $A$ 
is called the code of $\cD$ spanned by the blocks. 
 
Let $\cD=(X,\cB)$ be a design, and let $B \in \cB$ be a block 
of $\cD$.
The incidence structure
${\cD}^B$ = $(X',\cal{B'})$, 
where
\[ X'=B, \ {\cal{B'}}= \{ B\cap B_j \ | \ B_j \in {\cB}, B_j \neq B \}, \]
is called the {\it derived} design of $\cD$ with respect to block $B$.

The incidence structure
${\cD}_B=(X'', \cal{B''})$, 
 where
\[ X'' = X\setminus B, \ {\cal{B''}}=\{ B_j \setminus (B_j\cap B) \ | \  B_j \in {\cB}, B_j \neq B \}, \]
is called the  {\it residual} design of $\cD$ with respect to block $B$.

\begin{dfn}
\label{led}
Let $A$ be the $v$ by $b$ incidence matrix of a design $\cD$
with $v$ points and $b$ blocks, 
and let $A''$ be the incidence matrix
of a residual design ${\cD}_B$ with respect to a block $B$.
 The residual design ${\cD}_B$ is said to be {\bf linearly embeddable} 
\cite{Ton17} over $F=GF(p)$, ($p$ prime), if
\begin{equation}
\label{eq3}
\rank_p A = \rank_p A'' +1.
\end{equation}
\end{dfn}
\begin{note}
\label{cled}
A sufficient condition for a residual design ${\cD}_B$ to be linearly
embeddable is that the minimum distance of the linear code over $F$
spanned by the blocks of $\cD$ is equal to $|B|$ 
\cite[Theorem 2.2]{Ton17}.
\end{note}

Given a design $\cD=(X, \cB)$, its {\it dual}
design $\cD^*$ is the incidence structure having as points the blocks of
$\cD$, and having as blocks the points of $\cD$, 
where a point and a block
of $\cD^*$ are incident if and only if the corresponding block and point of
$\cD$ are incident.
If $A$ is the incidence matrix of $\cD$,
then $A^T$ is the incidence matrix of the dual design $\cD^*$.

Given integers $v\ge k \ge t \ge 0$, $\lambda\ge 0$, a $t$-$(v,k,\lambda)$ 
design (or briefly, a $t$-design)
 $\cD$
is an incidence
structure with $v$ points and blocks of size $k$ such that every $t$-subset
 of points is contained
in exactly $\lambda$ blocks. A $t$-$(v,k,\lambda)$ design is also 
an $s$-$(v,k,\lambda_s)$
design for every integer $s$ in the range  $0 \le s \le t$, with
$ \lambda_s = \lambda{ v-s \choose t-s }/{k-s \choose t-s}$. 

Let $\cD=(X, \cB)$ be a $t$-$(v,k,\lambda)$ design,
and let $x\in X$ be a point.
The {\it derived design} ${\cD}^{x}$  with respect  $x$
is the $(t-1)$-$(v-1,k-1,\lambda)$ design with point set $X\setminus \{ x \}$,
and block set $\{ B\setminus \{ x \} \ | \  B\in {\cB}, x\in B \}$. 

The {\it residual design}   ${\cD}_{x}$ with respect to $x$
is the $(t-1)$-$(v-1,k,\lambda_{t-1} - \lambda)$ design
with point set $X\setminus \{ x \}$,
and block set $\{ B \ | \  B\in {\cB}, x\notin B \}$.

If $\cD$ is a  2-$(v,k,\lambda)$ design with $v>k>0$,
the number of blocks $b=v(v-1)\lambda/(k(k-1))$
 satisfies the Fisher inequality:
\begin{equation}
\label{Fish}
 b \ge v, \end{equation}
and the equality $b=v$ holds if and only if every two blocks of $D$
share exactly $\lambda$ points.

A 2-$(v,k,\lambda)$ design $D$ with $b=v$ is called {\it symmetric}.
The dual design  $\cD^*$ of a symmetric 2-$(v,k,\lambda)$ design $\cD$
is a symmetric design having the same parameters as $\cD$.
A symmetric design is {\it self-dual} if it is isomorphic to its dual design.
A {\it biplane} is a symmetric design with $\lambda=2$.

\begin{note}
\label{n1}
Assume that  $\cD$ is a 2-$(v,k,\lambda)$ design.
If $x$ is a point of $\cD$,
the derived design ${\cD}^{x}$ is a 1-$(v-1,k-1,\lambda)$ design,
while the residual design ${\cD}_{x}$ is a 1-$(v-1,k,r-\lambda)$ design,
where $r = \lambda_1 =\lambda(v-1)/(k-1)$ is the number of blocks
of $\cD$ that contain $x$.

If  $\cD$ is a symmetric 2-$(v,k,\lambda)$ design, then $r=k$
and ${\cD}_{x}$ is a 1-$(v-1,k,k-\lambda)$ design.
In addition, if $B$ is a block, the derived design ${\cD}^B$
is a 2-$(k,\lambda,\lambda-1)$ design, while the residual design
${\cD}_B$ is a 2-$(v-k,k-\lambda,\lambda)$ design.
\end{note}

A 2-$(v,k,\lambda)$ design is {\it quasi-symmetric} with intersection
numbers $x$, $y$, ($x<y$), if every two blocks intersect in either $x$ or $y$ points.
A brief survey on quasi-symmetric designs is given in 
\cite{Mohan}. Links between quasi-symmetric designs and
error-correcting codes are discussed in \cite{Tcd}.

Examples of quasi-symmetric designs are:
(1) unions of identical copies of symmetric 2-designs;
(2) non-symmetric 2-$(v,k,1)$ designs; 
(3) strongly resolvable designs; (4) residual designs of biplanes.
A quasi-symmetric 2-$(v,k,\lambda)$ design with $2k < v$ which does not belong to any of 
these four classes is referred to as {\it exceptional}.

The classification of exceptional quasi-symmetric designs is a difficult
open problem. A table of  admissible parameters for exceptional
quasi-symmetric designs with number of points $v\le 70$ is
given in \cite[Table 48.25]{Mohan}. 
This table is an updated version of Neumaier's
table \cite{Neu} published in 1982.

Two admissible parameters sets for exceptional quasi-symmetric
designs whose existence  has been unknown since 1982,
are 2-$(56,12,9)$, ($x=0$, $y=3$), and 2-$(57,12,11)$, ($x=0$, $y=3$).
It is the goal of this paper to show that quasi-symmetric designs with these
parameters do not exist. 
Note that the nonexistence of the latter is a direct consequence of that of
former. We thank the anonymous reviewer who pointed out 
that the nonexistence of the latter is essentially known, if
one notices that 
the latter is a $3$-$(57,12,2)$ design. See Remark~\ref{rem:rev}.

The nonexistence of a quasi-symmetric 2-$(57,12,11)$ design
implies also the nonexistence of a quasi-3 design with 
parameters 2-$(267,57,12)$.
Similarly, using the nonexistence of quasi-symmetric 2-$(37,9,8)$ designs
with intersection numbers 1 and 3 \cite{HMT}, we show that a quasi-3
design with parameters 2-$(149,37,9)$ does not exist.
The existence of quasi-3 designs with these parameters
was a long standing open question \cite{McG}.

\section{Residual 2-$(45,9,2)$ designs and their ternary codes}

Our proof of the nonexistence of a quasi-symmetric 2-$(56,12,9)$ design
is based on the following lemmas.

\begin{lemma}
\label{lem1}
Suppose that $\cD=(X,\cB)$ is a  quasi-symmetric 2-$(56,12,9)$ design
with intersection numbers 0 and 3, and let $z\in X$ be a point of  $\cD$.
\begin{enumerate}[(i)]
\item The derived design ${\cD}^z$ is a 1-$(55,11,9)$ design with 45 blocks
whose dual design $({\cD}^z)^*$ is a 2-$(45,9,2)$ design.
\item The residual design ${\cD}_z$ is a 1-$(55,12,36)$ design with 165 blocks.
The columns of the $55\times 165$ incidence matrix of ${\cD}_z$ 
belong to the dual code
$C^\perp$ of the  linear code $C$ over $GF(3)$ spanned by the columns 
of the $55\times 45$ incidence matrix of ${\cD}^z$.
\end{enumerate}
\end{lemma}
 \begin{proof}
(i) Since every two non-disjoint blocks of $\cD$ 
share exactly three points, every two blocks of ${\cD}^z$
share exactly two points. This implies that $({\cD}^z)^*$ 
is a 2-$(45,9,2)$ design.

(ii) The inner product of the incidence vector of
every block $B$ of  ${\cD}_z$ with the incidence vector 
 of every block of  ${\cD}^z$ is either 0 or 3. 
\end{proof}

By a theorem of Hall and Connor \cite{HC},
every 2-$(45,9,2)$ design is a residual design
with respect to a block of a biplane with parameters 2-$(56,11,2)$. 
There are five nonisomorphic biplanes $Bi$, ($1 \le i \le 5$)
with these parameters \cite[15.8]{Hall}, all five being self-dual. 
The first biplane, $B1$, was found by Hall, Lane and Wales \cite{HLW},
$B2$ was found by Mezzaroba and Salwach \cite{SM},
$B3$ and $B4$ were found by Denniston \cite{Den},
and $B5$ was found by Janko and Trung \cite{JT}.
The residual 2-$(45,9,2)$ designs
of the five biplanes fall into 16 isomorphism classes
(see \cite[Table 2]{MR1700843}).

It was shown by an exhaustive computer search
(Kaski and \"Osterg\aa rd \cite{MR2384014}),
that up to isomorphism, there are exactly five biplanes
with 56 points, and consequently, exactly 16 nonisomorphic
2-$(45,9,2)$ designs.

Using Lemma~\ref{lem1}, the existence question for a
2-$(56,12,9)$ design can be resolved by computing the sets 
of all $(0,1)$-vectors of weight 12 in the dual codes of 
the ternary linear codes spanned by the
rows of the $45\times 55$  
incidence matrices of the 16 nonisomorphic 2-$(45,9,2)$ designs,
and checking if any of these sets contains a subset of 165 vectors 
that form the incidence matrix of a 1-(55,12,36) design with 165 blocks
of size 12, 
such that every two blocks are either disjoint or share exactly 
three points.


\begin{lemma}
\label{lemb}
Let ${\cD}_B$ be a 2-$(45,9,2)$ design with point set 
$\{ 1,2,\dots, 45 \}$, being the residual
design of a 2-$(56,11,2)$ biplane $\cD$ with respect
to a block $B=\{ 46,\dots, 56\}$. Let $A$ be an 
incidence matrix of $\cD$ given by \eqref{eqd},
where $A''$ is the
 $45\times 55$ incidence matrix of ${\cD}_B$,
and $A'$ is the $11\times 55$ incidence matrix of the derived 
2-$(11,2,1)$ design ${\cD}^B$.
\begin{equation}
\label{eqd}
A=
\begin{pmatrix} 
&0\\ 
A''&\vdots\\
&0\\
&1\\ 
A'&\vdots\\
&1\end{pmatrix}.
\end{equation}

If $c=(c_1,\dots,c_{55})$ is a $(0,1)$-codeword
of weight 12 in the dual code $(L'')^\perp$ of
the ternary linear code $L''$ spanned by the rows of $A''$,
then $c^*=(c_1,\dots,c_{55},0)$ belongs to the dual code $L^\perp$
of the ternary linear code $L$ spanned by the rows of $A$.
\end{lemma}
\begin{proof}
The ternary codes of the five biplanes with parameters 2-$(56,11,2)$ 
were computed in \cite{MR1700843}, and all five codes have minimum distance 11.
Thus, by Definition~\ref{led} and Note~\ref{cled}, every residual
2-$(45,9,2)$ design is linearly embeddable over $GF(3)$, and
\[ \rank_{3}A = \rank_{3}A'' +1. \]
This implies
\begin{equation}\label{LA''}
\dim L=
\rank_3\begin{pmatrix} 
&0\\ 
A''&\vdots\\
&0\\
\bar{1}_{55}&1
\end{pmatrix}
\end{equation}
Since the sum of all rows of $A$ over $GF(3)$ is the constant vector with all entries equal to 2,
the all-one vector $\bar{1}_{56}$ =$(\bar{1}_{55},1)$ of length 56 belongs to $L$.
Thus, $L$ is spanned by the row vectors of the matrix in the right-hand side of
\eqref{LA''}, to which $c^*$ is orthogonal. Therefore, $c^*\in L^\perp$.
\end{proof}

Lemma~\ref{lem1} and Lemma~\ref{lemb} imply the following (see Fig.~\ref{fig:lem5}
for an illustration).

\begin{lemma}
\label{lem5}
Let $\cD$ be a 2-$(56,11,2)$ biplane,
 and let $S$ be the set of all $(0,1)$-vectors
of weight 12 in the dual code of the ternary linear code
spanned by the points of $\cD$. 
 Let $\cD_B$ be a residual 2-$(45,9,2)$ design
of $\cD$ with respect to a block $B$.

A necessary condition for the existence of a quasi-symmetric
2-$(56,12,9)$ design $\hat{\cD}$ having ${\cD_B}^*$ as a derived design
is that the subset $S_B \subset S$ consisting of all vectors having 
0 in the position labeled by $B$, contains a set of 165 vectors
that are the incidence vectors of the blocks of 
a 1-$(55,12,36)$ design $\hat{\cD}_z$, such that  every two blocks of $\hat{\cD}_z$
are either disjoint or share exactly three points.
\end{lemma}

\begin{figure}
\begin{center}
\begin{tikzpicture}[scale=1] 
 \def\a{6}
 \def\b{2}
\draw[red,very thick] (0,1) -- (\a+1,1);
\draw[red,very thick] (0,1) -- (0,\b+3);
\draw[red,very thick] (0,\b+3) -- (\a+1,\b+3);
\draw[red,very thick] (\a+1,1) -- (\a+1,\b+3);
\draw[black,very thick] (1,0) -- (\a+2,0);
\draw[black,very thick] (1,0) -- (1,\b+1);
\draw[black,very thick] (\a+2,0) -- (\a+2,\b+1);
\draw[black,very thick] (1,\b+1) -- (\a+2,\b+1);
\draw[black] (\a+1,0) -- (\a+1,1);
\draw[black] (\a+1,1) -- (\a+2,1);
\draw[red] (0,\b+1) -- (1,\b+1);
\draw[red] (1,\b+3) -- (1,\b+1);
\node (DB) at (4,2) {\begin{tabular}{rcl} blocks & & points \\
 points \fbox{$\mathcal{D}_B$} & = &\fbox{$\mathcal{D}_B^*$} blocks
 \end{tabular}};
\node (B1) at (\a+1.5,0.5) {$\bar{1}$};
\node (B0) at (\a+1.5,2) {$\bar{0}$};
\node (DBB) at (4,0.5) {$\mathcal{D}^B$};
\node (DD1) at (0.5,2) {\color{red} $\bar{1}$};
\node (DD0) at (0.5,4) {\color{red} $\bar{0}$};
\node (11) at (-0.5,0.5) {$11$};
\node (45) at (-0.5,2) {$45$};
\node (165) at (-0.5,4) {$165$};
\node (SB) at (4,4) {\color{red} $\hat{\mathcal{D}}_z$};
\node (SBL) at (\a+4,4) {\color{red} rows of $\hat{\mathcal{D}}_z\subseteq S_B\subseteq S\subseteq L^\perp$};
\node (DL) at (\a+4,1.5) {rows of $\mathcal{D}$ spans $L$};
\node (z) at (0.5,-0.5) {{\color{red} $z$}};
\node (55) at (4,-0.5) {$55$};
\node (B) at (\a+1.5,-0.5) {$B$};
\end{tikzpicture} 
\end{center}
\caption{Lemma \ref{lem5}}
\label{fig:lem5}
\end{figure}

\section{The nonexistence of quasi-symmetric
$2$-$(56,12,9)$ and $2$-$(57,12,11)$ designs
}

\begin{theorem}\label{thm:main}
\begin{enumerate}[(i)]
\item 
\label{main1}
A quasi-symmetric $2$-$(56,12,9)$ design
with block intersection numbers $0, 3$ does not exist.
\item
\label{main2}
A quasi-symmetric $2$-$(57,12,11)$ design
with block intersection numbers $0, 3$ does not exist.
\end{enumerate}
\end{theorem}
\begin{proof}
(i) 
By a theorem of Hall and Connor \cite{HC},
every 2-$(45,9,2)$ design is a residual design
with respect to a block of a biplane with parameters 2-$(56,11,2)$. 
Thus, by Lemma~\ref{lem5},
it is sufficient 
to inspect the sets of all $(0,1)$-codewords of weight 12 in 
the dual codes of the ternary codes of the five biplanes with 56 points. 
The number of such codewords 
can be found by computing with Magma \cite{magma}
the complete weight enumerator of the dual codes.
This was done in \cite{MR1700843} for three of the biplanes,
$B1$, $B2$ and $B4$, while upper bounds 91 and 22 were found for
the codes of the biplanes $B3$ and $B5$  
\cite[Table 1]{MR1700843}, and these results
were used to prove that none of the five biplanes 
can be extended to a 3-$(57,12,2)$ design.

Using Magma, we were able to compute the
exact numbers of such codewords
 for the dual codes of $B3$ and $B5$ (84 and 20, respectively).
We reproduce some of the properties of these 
codes in Table~\ref{tab:1}, where the first column lists the 
corresponding biplane, column two gives the 3-rank of its incidence matrix, 
or the code
dimension, the third column gives the order of the automorphism group 
of the biplane,
the fourth column gives the minimum distance of the code,
 and the last column gives the total number of $(0,1)$-codewords
of weight 12 in the dual code.

Since the dual codes of the biplanes $B3$, $B4$ and $B5$ each contains less than 165
$(0,1)$-codewords of weight 12, it follows from
Lemma~\ref{lem5} that none of
the dual designs of their residual 2-$(45,9,2)$ designs 
can be a derived 1-$(55,11,9)$ design 
of a quasi-symmetric 2-$(56,12,9)$ design with intersection numbers $0, 3$.

Since the automorphism group of $B1$ acts transitively 
on the set of blocks, all residual
2-$(45,9,2)$ designs of $B1$ are isomorphic.
We define a graph $\Gamma_1$ having as vertices the 2100
$(0,1)$-codewords of weight 12 in the dual code
of the ternary code spanned by the points  of $B1$, 
where two codewords are adjacent in $\Gamma$ if 
their supports  are either disjoint or share exactly three 
points. It follows from Lemma~\ref{lem5} that if
a quasi-symmetric 2-$(56,12,9)$ design exists and has  
a derived design with respect to a point
which is the dual design of a residual 2-$(45,9,2)$ design
of $B1$, then $\Gamma_1$ will contain a clique of size 165. 
 A quick computation
with Cliquer \cite{cliquer} shows that the maximum clique size 
of  $\Gamma_1$ is 22, thus none of the residual designs of $B1$
is embeddable in a quasi-symmetric 2-$(56,12,9)$ design. 
 
 The graph $\Gamma_2$ having as vertices the 516 
$(0,1)$-codewords of weight 12 in the dual code
of the ternary code spanned by the points of of $B2$,
where two codewords are adjacent in $\Gamma_2$ if
their supports  are either disjoint or share exactly three
points, has maximum clique size 18. It follows by Lemma~\ref{lem5}
that a quasi-symmetric 2-$(56,12,9)$ design having a derived design
which is the dual design of some residual 2-$(45,9,2)$ design
of $B2$, does not exist.
This completes the proof of part (i).

(ii) Suppose there exists a quasi-symmetric $2$-$(57,12,11)$ design $\cD$
with block intersection numbers $0,3$. Then the residual
design of $\cD$ with respect to a point
$p$ is a quasi-symmetric $2$-$(56,12,9)$ design
with block intersection numbers $0,3$. 
Such a design does not exist by (i).
\end{proof}

\begin{rem}\label{rem:rev}
We thank the anonymous reviewer who pointed out 
that Theorem~\ref{thm:main}~(ii) is essentially known.
Indeed, a quasi-symmetric $2$-$(57,12,11)$ design
with block intersection numbers $0$ and $3$ is necessarily a 
$3$-$(57,12,2)$ design by using \cite[Propostion~12]{Neu}
(see also \cite{Cal}). The nonexistence of a
$3$-$(57,12,2)$ design has already been established in
\cite[Corollary~2]{MR2384014}.
\end{rem}

\begin{table}
\begin{center}
\begin{tabular}{|c|c|c|c|c|}
\hline
&$\dim C$&$|\Aut Bi|$&$\min $&\# wt 12 in $C^\perp$\\
\hline
B1&20&80640&11&2100\\
B2&22&288&11&516\\
B3&26&144&11&84${}^*$\\
B4&24&64&11&148\\
B5&26&24&11&20${}^*$\\
\hline
\end{tabular}
\caption{The ternary codes of the five biplanes. Entries with ${}^*$ 
were not found exactly, but were only estimated in \cite{MR1700843}}
\label{tab:1}
\end{center}
\end{table}

\section{The nonexistence of some quasi-3 designs}

\begin{dfn}
A symmetric 2-$(v,k,\lambda)$ design $D$ is a {\it quasi-3 design}
\cite{McG} with triple intersection numbers $x$ and $y$ ($x<y$) if
 every three blocks of $D$ intersect in either $x$ or $y$ points.
\end{dfn}

Clearly, $D$ is a quasi-3 design 
if and only if every of its derived
(with respect to a block) 2-$(k,\lambda,\lambda-1)$ designs
is a quasi-symmetric design with block intersection numbers $x$ and $y$.

According to \cite[Table 47.14]{McG}, there are 12 parameter sets
of quasi-3 designs with number of points $v \le 400$, for which the 
existence of a quasi-3 design is unknown. Two of these twelve open cases are
the parameters 2-$(149,37,9)$, ($x=1$, $y=3$), and
2-$(267,57,12)$, ($x=0$, $y=3$).

\begin{theorem}
\label{267}
A quasi-3 design with parameters 2-$(267,57,12)$, ($x=0$, $y=3$),
does not exist.
\end{theorem}
\begin{proof}
Any derived design with respect to a block of a
quasi-3 2-$(267,57,12)$ design with triple intersection numbers
$x=0$, $y=3$ is a quasi-symmetric 2-$(57,12,11)$ design with block intersection numbers
$x=0$, $y=3$.
By Theorem~\ref{thm:main}, part
(ii), a quasi-symmetric design with the latter parameters does not exist.
\end{proof}

\begin{theorem}
\label{149}
A quasi-3 design with parameters 2-$(149,37,9)$, ($x=1$, $y=3$), does not exist.
\end{theorem}
\begin{proof}
Any derived design with respect to a block of a 
quasi-3 2-$(149,37,9)$ design with triple intersection numbers $x=1$, $y=3$ 
is a quasi-symmetric 2-$(37,9,8)$ design with block intersection numbers
$x=1$, $y=3$. However, it was proved in \cite{HMT} that  quasi-symmetric
designs with the latter parameters do not exist.
\end{proof}

\end{document}